\documentclass[letterpaper, 10 pt, conference]{ieeeconf}  

\IEEEoverridecommandlockouts                              
\usepackage{amsmath} 
\usepackage{amssymb}  
\usepackage{tikz}
\usepackage{float}

\usepackage[bookmarks,bookmarksnumbered,colorlinks]{hyperref}
\hypersetup{linkcolor = black,anchorcolor = black,citecolor =
        black,filecolor = black,urlcolor = black}

\tikzstyle{node}=[thick,circle,draw=myblue,minimum size=22,inner sep=0.5,outer sep=0.6, fill=myblue!15]

\colorlet{myred}{red!80!black}
\colorlet{myblue}{blue!80!black}
\colorlet{mygreen}{green!60!black}
\colorlet{myorange}{orange!100}

\usepackage{color}
\usepackage{graphicx}

\def\N
{\mathbb{N}} 
\def\R{\mathbb{R}}
\def\Rn{\mathbb{R}^n}

\DeclareMathOperator{\dist}{\operatorname{dist_\mathcal{G}}} 
\DeclareMathOperator{\diag}{\operatorname{diag}} 

\newtheorem{theorem}{Theorem}
\newtheorem{example}[theorem]{Example}
\newtheorem{lemma}[theorem]{Lemma}
\newtheorem{proposition}[theorem]{Proposition}
\newtheorem{definition}[theorem]{Definition}

\newtheorem{remark}[theorem]{Remark}
\newtheorem{assumption}{Assumption}
\usepackage{comment}

\title{\LARGE \bf
Separable approximations of optimal value functions under a decaying sensitivity assumption
}

\author{Mario Sperl, Luca Saluzzi, Lars Grüne, and Dante Kalise 
\thanks{This research has been supported by the Deutsche Forschungsgemeinschaft (DFG, German Research Foundation) under project number 463912816 within the priority program 441826958.}
\thanks{M. Sperl and L. Grüne are with University of Bayreuth, Department of Mathematics, Germany {\tt\small mario.sperl@uni-bayreuth.de, lars.gruene@uni-bayreuth.de}}%
\thanks{L. Saluzzi and D. Kalise are with Imperial College London, Department of Mathematics, United Kingdom {\tt\small l.saluzzi@imperial.ac.uk, d.kalise-balza@imperial.ac.uk}}%
}

\begin{document}

\maketitle
\thispagestyle{empty}
\pagestyle{empty}

\begin{abstract}
An efficient approach for the construction of separable approximations of optimal value functions from interconnected optimal control problems is presented. The approach is based on assuming decaying sensitivities between subsystems, enabling a curse-of-dimensionality free approximation, for instance by deep neural networks.
\end{abstract}

\section{INTRODUCTION}
Optimal control problems arise in a variety of applications including robotics, aerospace, and power systems, among others. One way to approach such problems is by means of dynamic programming, i.e., by computing the optimal value function by solving a Hamilton-Jacobi-Bellman (HJB)-PDE. However, the computational complexity of grid-based numerical methods for this problem grows exponentially with respect to the state space dimension, a phenomenon known as the curse of dimensionality. Over the last years, the use of data-driven methods has emerged as a suitable alternative to circumvent this difficulty \cite{albi2022, dolgov2022data}. In \cite{nakamurazimmerer2021} the authors propose a deep learning method to approximate solutions of high-dimensional HJB-PDEs. The fact that neural networks are capable of overcoming the curse of dimensionality for compositional functions, see \cite{poggio2017} and \cite{kang2022}, has also been exploited in \cite{gruene2021_ComputingLyapunovFunctions} and \cite{gruene2022} for approximating separable (control) Lyapunov functions.

The main contribution of this paper is the construction of a separable approximation of the optimal value function of interconnected optimal control problems. The new idea presented in this paper is to use a decaying sensitivity assumption between subsystems with an increasing distance in a corresponding graph. We show how this decrease of sensitivity enables a separable approximation based on neighborhoods in the graph.

The notion of decaying sensitivity is strongly connected to the off-diagonal decay property studied by the numerical linear algebra community. A theoretical framework in this setting has been provided in \cite{benzi2016localization}, where the author proves decay bounds for the entries of general matrix functions and similar results are extended in \cite{haber2016} in an optimal control context, establishing decay properties for the solution of Lyapunov equations.

An exponential decay of sensitivity for nonlinear programs with a graph structure has been studied in \cite{shin2022expDecay}. Based on the results therein, a connection to dynamic optimization and to the particular case of discrete-time linear quadratic problems has been established in \cite{shin2021controllability} and \cite{shin2022lqr}, respectively. An exponential decay of sensitivity has also been discussed in the context of reinforcement learning and for linear quadratic problems in \cite{qu2022scalableReinforcement} and \cite{zhang2023optimal}, respectively. Note that the decaying sensitivity in \cite{shin2022lqr} and \cite{zhang2023optimal} is used for the purpose of an approximate optimal feedback matrix, while we want to approximate the optimal value function. 

\paragraph*{Notation}
For $j \in \N$ we denote $[j] := \lbrace 1, 2, \dots, j \rbrace$. We define $0_{n \times m}$ to be the zero matrix in $\R^{n \times m}$ and set $0_n := 0_{n \times n}$. The entries of a (block)-matrix $A$ are accessed via $A[\cdot, \cdot]$. The comparison space $\mathcal{L}$ denotes all continuous and strictly decreasing functions $\gamma \colon \R_{\geq 0} \to \R_{\geq 0}$ with $\lim_{\tau \to \infty} \gamma(\tau) = 0$. 

The remainder of this paper is organized as follows: In the next section, we introduce the problem. Afterwards, we use a decaying sensitivity assumption to construct separable approximations of optimal value functions. In Section \ref{sec.LQR}, we consider the particular case of discrete-time linear quadratic problems. Numerical test cases are presented in Section \ref{sec.NumExamples}.  

\section{Preliminaries}
\subsection{Optimal Control Problem} \label{subsec.Prelim}
We consider a network of $s \in \N$ agents or subsystems $\mathcal{V} = \lbrace v_1, \dots, v_s \rbrace$. Their interaction is represented via a directed graph $\mathcal{G} = (\mathcal{V}, \mathcal{E})$, where $(i,j) \in \mathcal{E}$ means that the agent $v_i$ influences agent $v_j$. In this case, we call $v_i$ a neighbor of $v_j$ and denote the index set of neighbors of $v_j$ by $	\mathcal{N}_j = \lbrace i \mid (i,j) \in \mathcal{E} \rbrace$. Each agent has a state $x_j \in \R^{n_j}$ and a control variable $u_j \in \R^{m_j}$. Moreover, we assign a cost functional $\ell_j$ to each agent. The vector of all states and the vector of all controls are denoted by $x \in \Rn$ and $u \in \R^m$, respectively, where $n = \sum_{j=1}^s n_j$ and $m = \sum_{j=1}^s m_j$. In total, we  consider an optimal control problem (OCP) either in continuous-time 
\begin{subequations} \label{PF.system}
	\begin{align} 
			\min \; & J(x_0, u) \notag \\ = &  \int_0^\infty e^{-\delta t}\sum_{j=1}^{s} \ell_j(x_j(t), u_j(t), x_{\mathcal{N}_j}(t), u_{\mathcal{N}_j}(t)) \, dt, \notag \\  
			\text{s.t.} \; & \dot{x}_j(t) = f_j(x_j(t), u_j(t), x_{\mathcal{N}_j}(t), u_{\mathcal{N}_j}(t)), \; j \in [s], \notag \\ 
			& x(0) = x_0,  \label{PF.System_cont}
	\end{align}
	or in discrete-time
	\begin{align}  
			\min & \, J(x_0, u) \notag \\ = & \sum_{k=0}^\infty \left(\frac{1}{1+\delta}\right)^{k+1} \sum_{j=1}^{s} \ell_j(x_j(k), u_j(k), x_{\mathcal{N}_j}(k),u_{\mathcal{N}_j}(k)), \notag \\  
			\text{s.t.} \; & x_j(k+1) = f_j(x_j(k), u_j(k), x_{\mathcal{N}_j}(k), u_{\mathcal{N}_j}(k)), \; j \in [s], \notag \\ 
			& x(0) = x_0, \label{PF.System_disc}
	\end{align}
\end{subequations}
where $\delta \in [0,\infty)$ is the discount factor, $x_{\mathcal{N}_j}$ and $u_{\mathcal{N}_j}$ comprise all states and controls belonging to $\mathcal{N}_j$, and $x_0 \in \Rn$ is the initial value. If we do not want to distinguish between the continuous-time formulation \eqref{PF.System_cont} and the discrete-time formulation \eqref{PF.System_disc}, we just refer to the system of interest as system \eqref{PF.system}. Let  $V \colon \Rn \to \R$ denote the optimal value function of the system \eqref{PF.system} given by $V(x) = \inf_u J(x,u)$. Our goal is to compute an approximation of $V$ on a hypercube $\Omega = [-a,a]^n$ for some $a \in \R_{>0}$, where we assume that $V$ is well-defined and finite on $\Omega$. More precisely, we want to compute an approximation that is a separable function.
\begin{definition} \label{def.Separable}
	Let $\Psi \colon D \subset \Rn \to \R$ and $d \in [n]$. We call $\Psi$ a $d$-separable function if there exist $ s \in [n]$ functions $\Psi^j \colon \R^{d_j} \to \R$ with $d_j \leq d$ for $j \in [s]$ such that 
	\begin{equation*}
		\Psi(x) = \sum\nolimits_{j = 1}^s \Psi^j(z_j), \quad x \in D, 
	\end{equation*}
	where $z_j = (x_{j_1}, x_{j_2}, \dots, x_{j_{d_j}})$ for some $j_i \in [n]$. 
\end{definition}
Note that in Definition \ref{def.Separable} we allow the same $x_i$ to appear in different vectors $z_j$ and $z_l$ with $j,l \in [s], j \neq l$. That is, we allow for overlapping domains of the functions $\Psi^j$. \\
The reason for our interest in separable functions is that under sufficient regularity assumptions, $d$-separable functions can be approximated by deep neural networks with a number of neurons that grows only polynomially in the state dimension $n$. This has been shown in \cite[Proposition 1]{gruene2021_ComputingLyapunovFunctions} for separable Lyapunov functions, but an extension to general separable functions is straightforward. For details regarding a curse-of-dimensionality-free approximation of separable functions with deep neural networks, we refer to \cite{poggio2017}, \cite{kang2022}, and \cite{gruene2021_ComputingLyapunovFunctions}. As a consequence, separable approximations of optimal value functions are of interest for the purpose of a curse-of-dimensionality-free computation of an approximate optimal value function.
\subsection{Problem Formulation}
We construct such a separable approximation $\Psi$ of an optimal value function $V$ under a decaying sensitivity assumption. The main idea is to define functions $\Psi^j$, $j \in [s]$, that model the influence of the substate $x_j$ on the optimal value function $V$. To this end, we rely on a decay of sensitivity between the subsystems. One way to formulate this assumption is to require that for $x \in \Omega$ and $i,j \in [s]$, $i \neq j$, the sensitivity 
\begin{equation} \label{PF.sensitivity}
	\delta_{ij}(x) = \frac{\partial (V(x) - V(x_1, \dots, x_{j-1}, 0, x_{j+1}, \dots, x_s))}{\partial x_i}
\end{equation}
decreases as the distance between nodes $v_i$ and $v_j$ increases. This decreasing property then allows us to define functions $\Psi^j$ on certain neighborhoods of $x_j$ in the graph and to neglect the influence of all subsystems outside these neighborhoods. Thus, the domain of each $\Psi^j$ is a lower-dimensional subspace of $\Rn$. This yields the separable form of our approximation. \\ 
In the following section, we construct the functions $\Psi^j$ formally and show that they indeed can be used to approximate the optimal value function under a suitable decay assumption. 

\section{Separable Approximation of the Optimal Value Function}\label{sec.SeparableApp}
Suppose we are given an OCP of the form \eqref{PF.system} with a corresponding graph $\mathcal{G}$. For $i, j \in [s]$ we define $\dist(i,j)$ to be the length of the shortest path in $\mathcal{G}$ from node $v_i$ to $v_j$. Further, for $j, l \in [s]$ we define 
\begin{equation} \label{def.graphNeighborhood}
	\mathcal{B}_l(j) := \lbrace i \in [s] \mid \dist(i,j) \leq l \rbrace 
\end{equation}
to be the neighborhood of node $v_j$ with distance $l$, that is, $\mathcal{B}_l(j)$ comprises the indices of all nodes from which $v_j$ can be reached within $l$ steps. Let $b_l^j := \sum_{i \in \mathcal{B}_l(j)} n_i$, i.e., the dimension of the subspace corresponding to $\mathcal{B}_l(j)$. In the following, we may want to project a vector $x \in \Rn$ onto such a subspace. To this end, for $\mathcal{B}_l(j) = \lbrace \beta_1, \dots, \beta_r \rbrace$ we define the matrix $H_l^{j} \in \R^{b_l^j\times n}$ blockwise by setting 
\begin{equation*}
	H_l^{j}[k,i] := \begin{cases}
		I_{n_i}, & \quad \text{if} \quad \beta_k = i, \\ 
		0_{n_{\beta_k} \times n_i}, & \quad \text{otherwise},  
	\end{cases}
\end{equation*}
for $k \in [r], i \in [s]$. Note that the matrix ${H_l^{j^\top}}$ embeds a vector $x_{\mathcal{B}_l(j)} \in \R^{b_l^j}$ into $\Rn$. Moreover, for $j \in \lbrace 0, \dots, s \rbrace$ we define the block-diagonal matrices 
\begin{equation*} 
	\Pi^{j} := \diag(0_{n_1}, \dots, 0_{n_{j}}, I_{n_{j+1}}, \dots, I_{n_s}) \in \R^{n \times n}.  
\end{equation*}
The matrix $\Pi^{j}$ maps the values of the first $j$ substates of some vector $x \in \Rn$ to $0$. Note that $\Pi^0 = I_n$ and $\Pi^s = 0_n$. Now we are in a position to define the mappings $\Psi_l^j \colon \R^{b_l^j} \to \R$ via 
\begin{equation} \label{def.Psi_j}
	x_{\mathcal{B}_l(j)} \mapsto V \big(\Pi^{j-1} {H_l^{j^\top}}x_{\mathcal{B}_l(j)} \big) - V\big(\Pi^{j} {H_l^{j^\top}} x_{\mathcal{B}_l(j)}\big). 
\end{equation}
Such a function $\Psi_l^j$ takes the substates of the neighborhood of distance $l$ of the node $v_j$ as an input, lifts the corresponding vector into the whole state space and then considers the difference (under $V$) of setting the first $j-1$ and the first $j$ substates to $0$, respectively. We illustrate the construction in \eqref{def.Psi_j} with the example of the semi-discrete heat equation: 
\begin{example}\label{ex.Heat}
	Consider the one-dimensional heat equation with diffusion coefficient $\sigma$, where the state space is discretized via centered finite differences with discretization parameter $\Delta x$. We add a distributed control and quadratic costs to the resulting semi-discrete equation. The cost functional is given as
	\begin{equation*}
		J(x_0, u) = \int_{0}^{\infty} x^\top(t) x(t) + u^\top(t) u(t) \, dt 
	\end{equation*}
	and the linear dynamics are 
	\begin{equation}
		\dot{x}(t) = \frac{\sigma}{(\Delta x)^2} \begin{bmatrix}
			-2 & 1 & 0 & \cdots  \\ 
			1 & -2 & 1 & \cdots  \\ 
			0 & 1 & -2 & \cdots  \\ 
			\vdots & \ddots & \ddots & \ddots  \\ 
		\end{bmatrix} x(t) + I_n u(t), 
	\end{equation}
	for constants $\sigma, \Delta x >0$. Then the corresponding graph $\mathcal{G}$ for $n=s=5$ has the form displayed in Figure \ref{fig.HeatNeighborhoods}. The sets $\mathcal{B}_1(1)$, $\mathcal{B}_1(2)$, and $\mathcal{B}_1(3)$ are circled in red, orange, and green, respectively. In this scenario, we have 
	\begin{align*}
		& \Psi_1^1(x_1, x_2) = V(x_1, x_2, 0, 0, 0) - V(0, x_2, 0, 0, 0), \\
		& \Psi_1^2(x_1, x_2, x_3) = V(0, x_2, x_3, 0, 0) - V(0, 0, x_3, 0,0), \\
	 	& \Psi_2^1(x_1, x_2, x_3) =  V(x_1, x_2, x_3, 0, 0) - V(0, x_2, x_3, 0, 0). 
	\end{align*}
  \vspace{-0.7cm}
	\begin{figure}[h]
		\centering
		\includegraphics[scale=0.75]{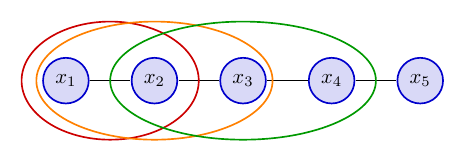}	
		\caption{Neighborhoods in the graph of the semi-discrete heat equation.}
  \vspace{-0.3cm}
		\label{fig.HeatNeighborhoods}
	\end{figure}
\end{example}
On the one hand, the use of $\Pi^{j}$ in \eqref{def.Psi_j} originates from the way we want to approximate the optimal value function $V$. The function $\Psi_l^1$ models the influence of the first substate on the optimal value function $V$. Hence, by using $\Psi_l^1$, we can reduce the computation of $V(x)$ to the computation of $V(0, x_2, \dots, x_n)$. We thus reduce the considered state space to $\lbrace 0 \rbrace^{n_1} \times \R^{n - n_1}$. Repeating this for $j = 2, \dots, s-1$, we finally arrive at the subspace $ \lbrace 0 \rbrace^{n-n_s} \times \R^{n_s}$, where we can use $\Psi_l^s$ to reduce the evaluation of $V(x)$ to the evaluation of $V$ at the origin. This procedure is illustrated geometrically for the case of $n=s=3$ in Figure \ref{fig.cube}. 
\begin{figure}[ht] \centering
    \includegraphics[scale=0.93]{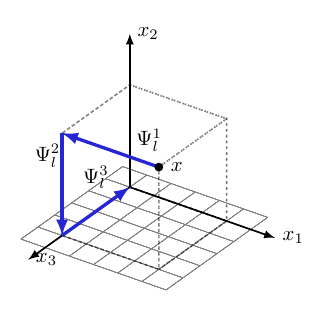}	
	\caption{Stepwise approximation of a point $x \in \R^3$ with the functions $\Psi_l^j$.}
    \vspace{-0.5cm}
	\label{fig.cube}
\end{figure}
On the other hand, the reason for using ${H_l^{j^\top}}$, i.e., for defining $\Psi_l^j$ on the subspace induced by $\mathcal{B}_l(j)$, is that we want the component functions of our approximation to be defined on lower-dimensional spaces, see Definition \ref{def.Separable}. To ensure that this local restriction approximates the global behavior, we rely on the following sensitivity assumption: 
\begin{assumption} \label{ass.Decay}
	There exists a function $\gamma \in \mathcal{L}$ such that 
    \begin{equation} \label{eq.Diff_V_and_V_j}
		\lvert \Psi_l^j(H_l^{j} x) - V(\Pi^{j-1} x) + V(\Pi^{j} x) \rvert \leq \gamma(l+1)  \tag{A1}
	\end{equation} 
 holds for all $l \in [s]$, $j \in [s]$, and $x \in \Rn$. 
\end{assumption}
The term $ V(\Pi^{j} x)-V(\Pi^{j-1} x)$ determines the influence of $x_j$ on $V$ in a global manner, i.e., with respect to all other subsystems. In \eqref{eq.Diff_V_and_V_j}, we compare this global difference with the local one obtained by $\Psi_l^j$. Note that in the case $l = s$, the left-hand side in \eqref{eq.Diff_V_and_V_j} is $0$. Now the statement of Assumption \ref{ass.Decay} is that with an increasing distance $l$ of the neighborhoods $\mathcal{B}_l(j)$, the error decreases as given by the $\mathcal{L}$-function $\gamma$. This means that the influence of nodes $v_i$ on the node $v_j$ decreases according to $\gamma$ with an increasing distance $\dist(i,j)$. In the case of a smooth optimal value function, Assumption \ref{ass.Decay} follows from the sensitivity inequality displayed in \eqref{PF.sensitivity}. 
\begin{lemma}\label{lem.Sensitivity}
    Consider an OCP of the form \eqref{PF.system} with a $\mathcal{C}^1$-optimal value function $V$. Define $\delta_{ij}$ as in \eqref{PF.sensitivity} and assume that there exists $\tilde{\gamma} \in \mathcal{L}$ such that for all $i,j \in [s]$, $i \neq j$, and $x \in \Omega$ it holds that 
    \begin{equation} \label{eq.SensitivityDecay}
        \lvert \delta_{ij}(x) \rvert \leq \tilde{\gamma}(\dist(i,j)). 
    \end{equation}
    Then Assumption \ref{ass.Decay} holds with $\gamma(\cdot) = s \max_{x \in \Omega} \lVert x \rVert \, \tilde{\gamma}(\cdot)$. 
\end{lemma}
\begin{proof}
	For $j \in [s]$ define $V_j \colon \Rn \to \R$ via 
	\begin{equation*}
		x \mapsto V(x) - V(x_1, \dots, x_{j-1}, 0, x_{j+1}, \dots, x_s).
	\end{equation*}
	By assumption, we have
	\begin{align*}
		\left \lVert \frac{\partial V_j}{\partial x_i} \right \rVert_{\infty, \Omega}\leq \tilde{\gamma}(\dist(i,j)), \quad i,j \in [s], i \neq j. 
	\end{align*}
Now fix some $j, l \in [s]$. Let $y := H_l^{j^\top} H_l^{j} x \in \Rn$. Using \eqref{def.Psi_j} and the convexity of $\Omega$, we obtain 
	\begin{align*}
		& \lvert \Psi_l^j(H_l^{j} x) - V(\Pi^{j-1} x) + V(\Pi^{j} x) \rvert \\ 
		= \, &  \lvert V_j(\Pi^{j-1} y) - V_j( \Pi^{j-1} x) \rvert \\ 
		\leq \, & \lVert \Pi^{j-1}(x - y) \rVert \int_0^1 \lVert \nabla V_j(\Pi^{j-1}((1-\tau) x + \tau y)) \rVert \, d\tau \\
		\leq \, & \lVert x - y \rVert \, \sum_{i \notin \mathcal{B}_l(j)} \left \lVert \frac{\partial V_j}{\partial x_i} \right \rVert_{\infty, \Omega} 
		\leq s \max_{x \in \Omega} \lVert x \rVert \, \tilde{\gamma}(l+1), 
	\end{align*}
 where we used that $x-y \in \Omega$ and $\Pi^{j-1} \Omega \subset \Omega$. 
\end{proof}
Note that Lemma \ref{lem.Sensitivity} is independent of the structure of the graph corresponding to the OCP. In particular, the last estimate in the proof considers the case where for all nodes $i \notin \mathcal{B}_l(j)$ we have $\dist(i,j) = l+1$. Depending on the graph-structure and the function $\tilde{\gamma}$, we may be able to avoid the dependence on the number of subsystems $s$ in the last estimate of the proof of Lemma \ref{lem.Sensitivity}. For instance, consider a sequential graph as in Example \ref{ex.Heat}, where $\dist(i,j) = \lvert i - j \rvert$, and assume an exponential decay given through $\tilde{y}(l) = \rho^l$ for some $\rho \in (0,1)$. Then we can modify the last estimate in the proof of Lemma \ref{lem.Sensitivity} by 
\begin{align*}
    & \lVert x - y \rVert \, \sum_{i \notin \mathcal{B}_l(j)} \left \lVert \frac{\partial V_j}{\partial x_i} \right \rVert_{\infty, \Omega} \\ 
    \leq \, & \lVert x - y \rVert \left( \sum_{i=1}^{j-l-1} \rho^{j-i} + \sum_{i=j+l+1}^{s} \rho^{i-j} \right) \leq  \frac{2\lVert x - y \rVert}{1-\rho}. 
\end{align*}

The following theorem shows that the functions $\Psi_l^j$ defined in this section together with Assumption \ref{ass.Decay} indeed allow to construct an approximation of the optimal value function. 
\begin{theorem} \label{thm.SepApprox}
	Consider an OCP of the form \eqref{PF.system} and let Assumption \ref{ass.Decay} hold. Then for all $l \in [s]$ and $x \in \Omega$ it holds \begin{equation} \label{eq.ApproxV}
		\lvert V(x)  - \sum_{j=1}^s \Psi_l^j(H_l^{j} x) - V(0) \rvert \leq (s-1) \gamma(l+1).
	\end{equation}
\end{theorem}
\begin{proof}
	We prove the claim by induction over $s$. If $s=1$, we have $\mathcal{B}_l(1) = \lbrace 1 \rbrace$. Thus, from \eqref{def.Psi_j} we get
	\begin{equation*}
		\Psi_1^1(x_1) = V(x_1) - V(0),  
	\end{equation*}
	whence \eqref{eq.ApproxV} follows. Now assume the assertion holds for $s-1$. First, observe that 
	\begin{align*}
		& \lvert V(x) -  \sum_{j=1}^s \Psi^j_l(H_l^{j} x) - V(0) \rvert \\ \leq & \lvert V(x) - \Psi_l^1(H_l^{1} x) - V(\Pi^{1}x) \rvert \\ & +  \lvert V(\Pi^{1}x) - \sum_{j=2}^s \Psi^j_l(H_l^{j} x) - V(0) \rvert \\ \leq \, & \gamma(l+1) + \lvert V(\Pi^{1} x) - \sum_{j=2}^s \Psi^j_l(H_l^{j} x) - V(0) \rvert, 
	\end{align*}
	by using \eqref{eq.Diff_V_and_V_j} for $j=1$. Note that in the evaluation of $V(\Pi^{1} x)$ as well as in all evaluations $\Psi^j_l(H_l^{j} x)$ for $j \geq 2$, the value of $x_1$ is set to $0$, i.e., the respective functions operate on the subspace $\Pi^{1} \Rn$. Thus, we can use the isomorphism $\Pi^{1} \R^{n} \cong \R^{n-n_1}$ to interpret the expression 
    \begin{equation*}
        \lvert V(\Pi^{1} x) - \sum_{j=2}^s \Psi_l^j(H_l^{j} x) - V(0) \rvert
    \end{equation*}
    as an equation of the form \eqref{eq.ApproxV} on $\R^{n-n_1}$ with $s-1$ substates. Since \eqref{eq.Diff_V_and_V_j} holds for all $j = 2, \dots, s$, we can use the induction hypothesis to obtain
	\begin{equation*}
		\lvert  V(\Pi^1 x) - \sum_{j=2}^s \Psi_l^j(H_l^{j} x) - V(0) \rvert \leq (s-2) \gamma(l+1). 
	\end{equation*}
	This completes the proof.  
\end{proof}
Theorem \ref{thm.SepApprox} states that we can approximate the optimal value function $V$ with the function $\Psi := \sum_{j=1}^s \Psi_l^j + V(0)$. This is a $d$-separable function for $d := \max_{j \in [s]} b_l^j$. 
\begin{remark}
   As depicted in Figure \ref{fig.cube}, the functions $\Psi_l^j$ allow for an approximation of $V$ on the whole cube $\Omega$ under the knowledge of the evaluation of $V$ at $0$. However, there is no need to fix the origin for this evaluation, but one may use any vector $\omega \in \Omega$. By defining affine-linear mappings $\tilde{\Pi}^j \colon \Rn \to \Rn$, $j \in \lbrace 0, \dots, s \rbrace $, via 
   \begin{equation*}
        x \mapsto \Pi^j x + \begin{bmatrix}
            \omega_1 & \dots & \omega_j & 0 & \dots & 0
        \end{bmatrix}^\top, 
   \end{equation*}
   and by replacing $\Pi^j$ with $\tilde{\Pi^j}$ in the definition of $\Psi_l^j$ in \eqref{def.Psi_j}, one obtains an approximation of the form $V \approx \sum_{j=1}^s \Psi_l^j + V(\omega)$. The choice of $\omega$ may depend on the knowledge of (an approximation of) $V$ at $\omega$. In many cases, e.g., for linear-quadratic problems as discussed in Section \ref{sec.LQR}, one has that $V(0)=0$. For the sake of a simple presentation, we stick to the case $\omega = 0$ for the rest of this paper.  
\end{remark} 
\begin{remark}
    Suppose that we want to approximate $V$ with $\Psi$ satisfying a given error bound $\varepsilon \in \R_{>0}$. Theorem \ref{thm.SepApprox} yields that this can be ensured by choosing $l$ large enough, so that 
    \begin{equation} \label{eq.errorBound}
        \gamma{(l+1)} \leq \frac{\varepsilon}{(s-1)}.
    \end{equation}
    Since the number of neurons in the neural network grows exponentially with $d$, we have to ensure that $d$ does not grow too fast with $s$ in order to avoid the curse of dimensionality. This is possible, e.g., if $\gamma(l)$ decreases exponentially in $l$. More precisely, if $\gamma(l) = \rho^l$ for some $\rho \in (0,1)$, one has to choose $l \geq \log_{\rho}(\frac{\varepsilon}{s-1}) - 1$. Now consider an increasing number of subsystems $s$. If the dimensions $b_l^j$ of the subspaces belonging to the neighborhoods $\mathcal{B}_l(j)$ are increasing subexponentially with $s$ and a growth of $l$ that ensures \eqref{eq.errorBound}, we can avoid the curse of dimensionality in the computation of the functions $\Psi_l^j$. 
\end{remark}
\begin{remark}
	The definition of the mappings $\Psi_l^j$ in \eqref{def.Psi_j} uses overlapping domains. This is in contrast to the works \cite{gruene2021_ComputingLyapunovFunctions} and \cite{gruene2022}, where a small-gain condition was used to prove the existence of a separable (control) Lyapunov function of the form 
	\begin{equation} \label{eq.SmallGainSeparable}
		V(x) = \sum_{j=1}^s V_j(x_j). 
	\end{equation}
	If $V$ can be written in the form \eqref{eq.SmallGainSeparable}, then Assumption \ref{ass.Decay} is trivially satisfied, since 
	\begin{equation*}
		\Psi_l^j(x_{\mathcal{B}_l(j)}) = V_j(x_j) - V_j(0) = V(\Pi^{j-1}x) - V(\Pi^{j} x), 
	\end{equation*}
	whence the left-hand side in \eqref{eq.Diff_V_and_V_j} is always $0$. Thus, one may interpret the present construction with overlapping domains as a method to allow for more general decompositions if there is no representation of $V$ as in \eqref{eq.SmallGainSeparable}. However, as a trade-off we no longer have an exact representation of the optimal value function in a separable form. 
\end{remark}


\section{Discrete-Time LQR} \label{sec.LQR}
In this section, we examine when our assumptions are satisfied for undiscounted discrete-time systems of the form \eqref{PF.System_disc} with linear dynamics and quadratic costs. It is known that one can compute the solution $P \in \R^{n \times n}$ of the discrete-time algebraic Riccati equation to obtain the optimal value function $V(x) = x^\top P x$. Thus, there is no need to use the proposed separable approximation technique to compute $V$. However, linear quadratic discrete-time problems allow for an explicit check whether Assumption \ref{ass.Decay} holds. Thus, we use this class of systems to argue that our assumptions are reasonable and to illustrate the computations. The corresponding OCP takes the form 
\begin{align} \begin{split}  \label{eq.LQR_discrete}
		\min \; & J(x_0, u) = \sum_{k=0}^\infty x(k)^\top Q x(k) + u(k)^\top R u(k),  \\  
		\text{s.t.} \; & x(k+1) = Ax(k) + Bu(k), \\ 
		& x(0) = x_0, 
\end{split} \end{align}
where $A \in \R^{n \times n}$, $B \in \R^{n \times m}$, $Q \in \R^{n \times n}$ is positive semidefinite and $R \in \R^{m \times m}$ is positive definite. We assume that $(A,B)$ is stabilizable, $(A,Q^{\frac{1}{2}})$ is detectable, and that $\mathcal{G}$ is undirected, i.e., we have $(i,j) \in \mathcal{E}$ if and only if $(j,i) \in \mathcal{E}$. Note that due to the definition of $\mathcal{G}$ for some $i,j \in [s]$ with $\dist(i,j) > 1$ it holds that all block matrices $A[i,j]$, $B[i,j]$, $Q[i,j]$, and $R[i,j]$ are zero. For this setting, an exponential decay property for the optimal feedback matrix $K = (B^\top P B + R)^{-1} (B^\top P A)$ has been shown in \cite{shin2022lqr}.
\begin{definition}[cf. Def. 1 in \cite{zhang2023optimal}]
Let $\mathcal{G}$ be a graph with $s$ vertices. We say that a matrix $X \in \mathbb{R}^{n \times n}$ given by block-matrices $X[i,j]$, $i,j \in [s]$, is $(C_X,\rho_X)$-spatially exponential decaying (SED) with respect to $\mathcal{G}$ for some $C_X \in \R_{\geq 0}$ and $ \rho_X \in (0,1)$ if for all $i,j \in [s]$ it holds that
$$
|X[i,j]| \le C_X \rho_X^{\dist(i,j)}. 
$$
\end{definition}
Since we are considering a fixed graph $\mathcal{G}$ corresponding to \eqref{eq.LQR_discrete}, we omit the dependence on $\mathcal{G}$ in the following. Under suitable assumptions it has been shown in \cite{shin2022lqr} that there exist $C_K \in \R_{\geq 0}$ and $\rho_K \in (0,1)$ independently of $s$ such that the optimal feedback matrix $K$ is $(C_K, \rho_K)$-SED. For details, we refer to the Chapters 3 and 5 and in particular to Theorem 3.3 in \cite{shin2022lqr}. We now discuss how this decay property yields Assumption \ref{ass.Decay}. Let $l \in [s]$ and define the matrix $P^l \in \R^{n \times n}$ blockwise by
\begin{equation} \label{def.Pl}
    P^l[i,j] = \begin{cases}
       & P[i,j], \quad \text{if} \, \dist(i,j) \leq l, \\
      & 0_{n_i \times n_j}, \quad \text{otherwise}. 
    \end{cases}
\end{equation}
Then the quadratic form $x^\top P^l x$ equals the approximation constructed in Section \ref{sec.SeparableApp}, as the following lemma shows. 
\begin{lemma} \label{lem.QuadraticApprox}
    Consider an OCP of the form \eqref{eq.LQR_discrete} and let $l \in [s]$. Define the functions $\Psi_l^j$ as in \eqref{def.Psi_j}. Then it holds that 
    \begin{equation*}
        \sum_{j=1}^s \Psi_l^j(H_l^{j} x) = x^\top P^l x, \quad x \in \Omega.  
    \end{equation*}
\end{lemma}
\begin{proof}
    Let $j \in [s]$ be fixed. We have 
    \begin{align*}
        & \Psi_l^j(H_l^{j} x) =  V\big(\Pi^{j-1} {H_l^{j^\top}} H_l^{j} x \big) - V\big(\Pi^{j} {H_l^{j^\top}} H_l^{j} x \big) \\ 
        = \, & x^\top {H_l^{j^\top}} H_l^{j} \, \big( \Pi^{j-1^\top} P \Pi^{j-1} - \Pi^{j^\top} P \Pi^{j} \big) \,  {H_l^{j^\top}} H_l^{j} x. 
    \end{align*}
    Observe that $\Pi^{j-1^\top} P \Pi^{j-1} - \Pi^{j^\top} P \Pi^{j} $ yields all blocks $P[i,j]$ for $i \geq j$ and the matrix ${H_l^{j^\top}} H_l^{j} \in \R^{n \times n}$ then eliminates all blocks with $\dist(i,j) > l$. Thus, we get 
    \begin{equation*} \label{eq.PsiQuadratic}
        \Psi_l^j(H_l^{j} x) = x_j^\top P[j,j] x_j + \sum_{\substack{i > j, \\ \dist(i,j) \leq l}} 2 x_j^\top P[j,i] x_i. 
    \end{equation*}
    Finally, summing up over $j=1, \dots, s$ and using \eqref{def.Pl} completes the proof. 
\end{proof}
Lemma \ref{lem.QuadraticApprox} shows that a decay of sensitivity in the matrix $P$ is a sufficient condition for Assumption \ref{ass.Decay}. In the following, we prove that for a sequential graph as in Example \ref{ex.Heat} the matrix $P$ is SED. To this end, we build on the SED-property of the matrix $K$ that has been shown in \cite{shin2022lqr}. Due to space limitations, the proof can be found in the arXiv version of the paper, see \cite{sperl2023separable}. 
\begin{proposition} \label{prop.DecayP}
    Consider an OCP of the form \eqref{eq.LQR_discrete} and assume that for $i,j \in [s]$ we have $\dist(i,j) = | i - j |$. Then the solution $P$ of the discrete-time algebraic Riccati equation is $(C_P, \rho_P)$-SED for constants $C_P$, $\rho_P$ independent of $s$.  
\end{proposition}
\begin{proof} Since $V(x) = x^T P x$, by inserting the optimal control $u(k) = K x(k)$ and the corresponding trajectory $x(k) = A_{cl}^k x(0)$ into the cost functional, we get $$
P = \sum_{k=0}^{\infty} (A_{cl}^k)^\top D A_{cl}^k, 
$$
where $A_{cl} = A-BK$ and $D=Q + K^\top R K$. According to Theorem A.7 in \cite{shin2022lqr}, the following estimate holds
\begin{equation}
|A^k_{cl}[i,j] |\le \lVert A^k_{cl} \rVert_2 \le \Upsilon \rho^k_{cl}, 
\label{cl1}
\end{equation}
where $\Upsilon \ge 1$ and $\rho_{cl} \in (0,1)$ are independent of $s$.
Since $ \dist(i,j) = |i-j|$ for all $i,j \in [s]$ it holds that $A[i,j]$, $B[i,j]$, $Q[i,j]$, and $R[i,j]$ are zero if $|i-j| > 1$. By exploiting this property, one can calculate that 
\begin{equation*}
|A_{cl}[i,j]| \le (\Vert A \Vert_{2}  + 3 C_K \Vert B \Vert_2 ) \rho^{-1}_{K} \rho_K^{|i-j|} =: C_{cl} \rho_K^{|i-j|}.
\end{equation*}
Together with the estimate \eqref{cl1} we obtain
\begin{equation} \label{cl3}
|A^k_{cl}[i,j]| \le \Upsilon C_{cl} \rho_K^{|i-j|} \rho^{k-1}_{cl} , \quad k \ge 1.
\end{equation}
Moreover, by straightforward computations we obtain for $\alpha \in (0,1]$ that $D$ is a $(C_D, \rho_K^\alpha)$-SED matrix, where 
$$
C_D =\Vert Q \Vert_{2} \rho^{-\alpha}_{K}+  \Vert R \Vert_{2} C_1(\alpha,s,i,j),
$$
$$
C_1(\alpha,s,i,j)=\sum_{r=1}^s \sum_{p=r-1}^{r+1} \rho_K^{|r-i|+|p-j|-\alpha|i-j|}. 
$$
Applying \eqref{cl3} and the SED-property of $D$ yields 
\begin{align*}
    & |P[i,j]| \le \sum_{k=0}^{\infty} \sum_{r,p = 1}^s  |A^k_{cl}[r,i] D[r,p] A^k_{cl}[p,j] | \le C_D \rho_K^{\alpha|i-j|} \\ & +  \Upsilon^2 C^2_{cl} C_D \sum_{k=1}^{\infty} \rho^{2(k-1)}_{cl} \sum_{r,p = 1}^s \rho_K^{|r-i|+\alpha|r-p|+|p-j|}. 
\end{align*}
Since $\rho_{cl} <1$, we can compute the infinite time series 
$$
\sum_{k=1}^{\infty} \rho^{2(k-1)}_{cl} = \frac{1}{1-\rho^2_{cl}}.  
$$
Finally, we obtain for $\beta \in (0,1]$ that $P$ is $(C_P, \rho_K^\beta)$-SED with
$$
C_P =  \frac{\Upsilon^2 C^2_{cl} C_D}{1-\rho^2_{cl}}  C_2(\alpha,\beta,s,i,j) + C_D,
$$
\begin{equation}
 C_2(\alpha,\beta,s,i,j) = \sum_{r,p = 1}^s \rho_K^{|r-i|+\alpha|r-p|+|p-j|-\beta|i-j|}. 
 \label{C2}
\end{equation}
It is left to show that $C_1$ and $C_2$ are bounded independently of the number of subsystems $s$. Since $P$ is symmetric, we may assume $i \geq j$. Write $i = j + \kappa s$ for $\kappa \in (0,1)$. Now consider the mappings 
\begin{align*}
    & \kappa \mapsto C_1(1,s,j + \kappa s, j), \\
    & \kappa \mapsto C_2(1,1,s, j + \kappa s, j). 
\end{align*}
The derivatives of these two mappings are strictly negative for $\kappa \in (0,1)$, leading to 
\begin{align*}
& C_1(1,s,i,j) < C_1(1,s,s,1), \quad (i,j) \neq (s,1), \\
& C_2(\alpha,1,s,i,j) <  C_2(\alpha,1,s,s,1), \quad (i,j) \neq (s,1).
\end{align*}
By continuity of $C_1, C_2$ there exist $\delta \in \R_{>0}$ such that for $\alpha, \beta \in (1-\delta, 1]$ it holds that
\begin{align*}
& C_1(\alpha,s,i,j) < C_1(\alpha,s,s,1),\quad (i,j) \neq (s,1),\\
& C_2(\alpha,\beta,s,i,j) <  C_2(\alpha,\beta,s,s,1), \quad (i,j) \neq (s,1).
\end{align*}
Hence, it is sufficient to prove a bound for $C_1$ and $C_2$ for $i = s$ and $j=1$ in the prescribed neighborhood of $\alpha$ and $\beta$. It holds that  
\begin{align*}
  C_1(\alpha,s,s,1) & = \rho_K^{(1-\alpha)s+\alpha-1} \sum_{r=1}^s  \rho_K^{-r}\sum_{p = r-1}^{r+1} \rho_K^{p} \\ 
  & \le 3 s \rho_K^{(1-\alpha)s+\alpha-2}
\end{align*}
is bounded for all $s \in \N$ if $\alpha<1$. Furthermore, we have  
\begin{align*}
     & C_2(\alpha,\beta,s,s,1) =\rho_K^{(1-\beta)s +\beta-1} \sum_{r = 1}^s  \rho_K^{-r}\sum_{p = 1}^s \rho_K^{\alpha|r-p|+p} \\
      & = \rho_K^{(1-\beta)s +\beta-1} \sum_{r = 1}^s  \rho_K^{-r} (\sum_{p=r}^s \rho_K^{p(1+\alpha)-\alpha r} + \sum_{p=1}^{r-1} \rho_K^{p(1-\alpha)+\alpha r}) \\
      & \le \rho_K^{(1-\beta)s +\beta-1}  \sum_{r = 1}^s \bigg( \rho_K^{-r(1+\alpha)} \frac{\rho_K^{r(1+\alpha)}}{1-\rho_K^{1+\alpha}} \\ 
       & \quad + \rho_K^{-r(1-\alpha)} \frac{\rho_K^{(1-\alpha)} - \rho_K^{(1-\alpha)r}}{1-\rho_K^{1-\alpha}} \bigg) \\
       & \le \rho_K^{(1-\beta)s +\beta-1} s \bigg(\frac{1}{1-\rho_K^{1+\alpha}}+ \frac{\rho_K^{(1-\alpha)(1-s)} - 1}{1-\rho_K^{1-\alpha}}\bigg) \\ 
       & \leq \rho_K^{s (\alpha - \beta)} \eta(s), 
\end{align*}
where $\eta$ is linear in $s$. Thus, $C_2$ is bounded uniformly for $s \in \mathbb{N}$ if $1>\alpha>\beta$. This concludes the proof.
\end{proof}
Note that due to the assumption $\dist(i,j) = |i-j|$ in Proposition \ref{prop.DecayP}, the matrices $A$, $B$, $Q$, and $R$ are block-tridiagonal. The proof can be extended for general $r$-banded matrices. This then corresponds to the assumption that $\dist(i,j) > 1$ for $|i-j| > r$, i.e., a sequential graph where every node $v_j$ might be connected with up to $2r$ nodes. In the following section, we present numerical test cases showing this decaying sensitivity property. 

\section{Numerical Examples} \label{sec.NumExamples}

\subsection{Random Linear Quadratic Regulator}
In this example we deal with a generic LQR problem with pseudo-random matrix $A$ generated by random uniform samples in $(0,1)$ by using the MATLAB command \texttt{rand}. We fix $B = Q = R = I_s$ and $s=100$. Starting from a pseudo-random matrix $A \in \mathbb{R}^{s \times s}$ taking values in the interval $[0,1]$, we define a family of $r$-banded matrices $\{A_r\}_{r=0}^n$ by
\begin{equation}
(A_r)[i,j] = \begin{cases}
  A[i,j], & |i-j| \le r,\\
  0, & |i-j| > r.
\end{cases}
\label{banded}
\end{equation}
Fixing the parameter $r \in \{1,\ldots,n\}$, we denote by $P_r$ the solution of the corresponding DARE. The Riccati Equation is solved by the MATLAB implicit solver \texttt{idare}. In Figure \ref{Fig_lqr1} we display a logarithmic plot of the absolute value of the first column of the matrix $P_r$  considering different bandwidths for the matrices $A_r$. In general, we notice a decay for all the choices of the parameter $r$ and the lower is the band-with, the steeper is the descent. 


\begin{figure}[h]
\centering
	\includegraphics[scale=0.44]{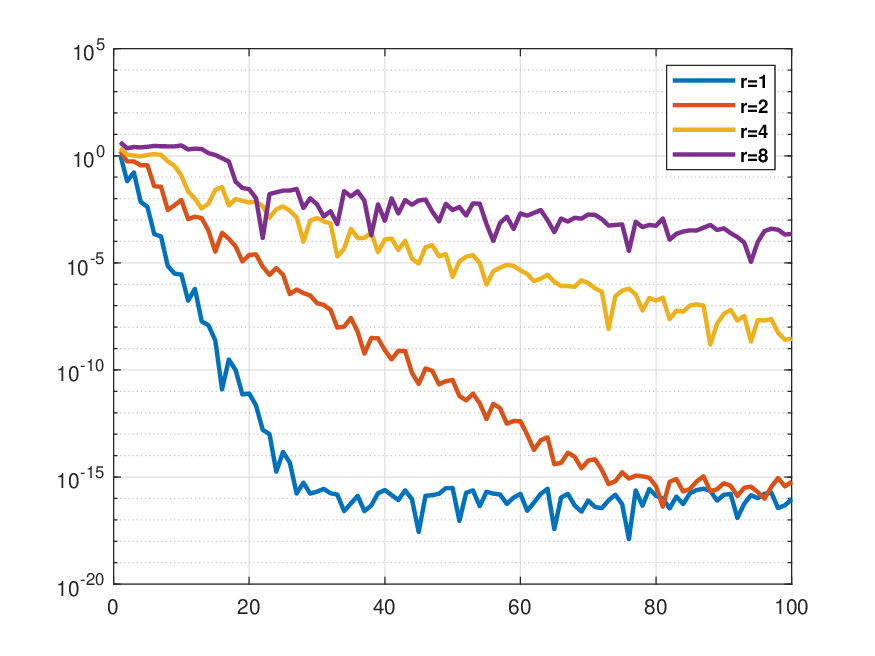}		
  \vspace{-0.6cm}
\caption{Decay in logarithmic scale of the first column of the absolute value of solution of the Riccati equation varying the parameter $k$.}
\label{Fig_lqr1}
\end{figure}
\subsection{Nonlinear example: Allen-Cahn equation}
In the second example we consider the following Allen-Cahn PDE with polynomial nonlinearity and homogeneous Neumann boundary conditions:
\begin{equation}
\left\{ \begin{array}{l}
\hspace{-3 pt} \partial_t y(t,x) = \sigma \partial_{xx} y(t,x) + y(t,x) (1-y(t,x)^2) + u(t,x),  \\
 y(0,x)=y_0(x),
\end{array} \right.
\label{AC}
\end{equation}
with $x \in [0,1]$ and $t \in (0,+\infty)$. Our aim is to steer the solution to the unstable equilibrium $\tilde{y}(x) =0$ minimizing the following cost functional
$$
J(u,y_0) = \int_0^{\infty}  \int_0^1 (|y(t,x)|^2 +   |u(t,x)|^2) dx \, dt \,.
$$
Discretizing the PDE \eqref{AC} by finite difference using $s=100$ grid points, we compute the feedback control by means of the State Dependent Riccati Equation (SDRE).  This approach is an extension of the LQR which takes into account the nonlinear terms present in the cost functional and in the dynamical system. More precisely, writing the nonlinear dynamics in semilinear form
$$
\dot{y}(t) = A(y(t))y(t) +B u(t),
$$
and given the discrete cost functional
$$
J = \gamma \int_0^{\infty} y^\top(t) y(t) + u^\top(t) u(t) dt,
$$
with $\gamma = 1/s$, the feedback control will be given by
\begin{equation}
u(y) = - \gamma^{-1} B^\top P(y),
\label{feedback_sdre}
\end{equation}
where $P(y)$ solves the SDRE

\begin{equation}
A^{\top}(y) P(y) + P(y) A(y) - P(y) B \gamma^{-1} B(y)^{\top} P(y)+\gamma I=0.
\label{SDRE}
\end{equation}

The SDRE feedback generates a stationary Hamiltonian with respect to $u$ and, as the state is driven to zero, necessary optimality conditions are asymptotically satisfied at a quadratic rate. We refer to \cite{ccimen2008state} for an extensive discussion of the topic.
Given the solution $P(y_0)$ of the SDRE \eqref{SDRE} computed for the initial condition $y_0= [\sin(\pi x_i)]_{i=1}^{n}$, in Fig. \ref{Fig_sdre} we show the decay of the absolute values of the first column of SDRE solution varying the viscosity parameter $\sigma \in \{10^{-k},\; k=1,\ldots,4\}$. We note that a lower value of the viscosity corresponds to a steeper decay in the sensitivity. This is a reasonable behaviour since increasing the viscosity, the transmission speed is increasing as well and particles will influence each other faster. Compared to our first numerical test, we observe a different decay behaviour: in the first example it was related to the bandwidth of the matrix $A$, while in this test it is connected to the propagation speed of the diffusion. 
\begin{figure}[h]
\centering
	\includegraphics[scale=0.44]{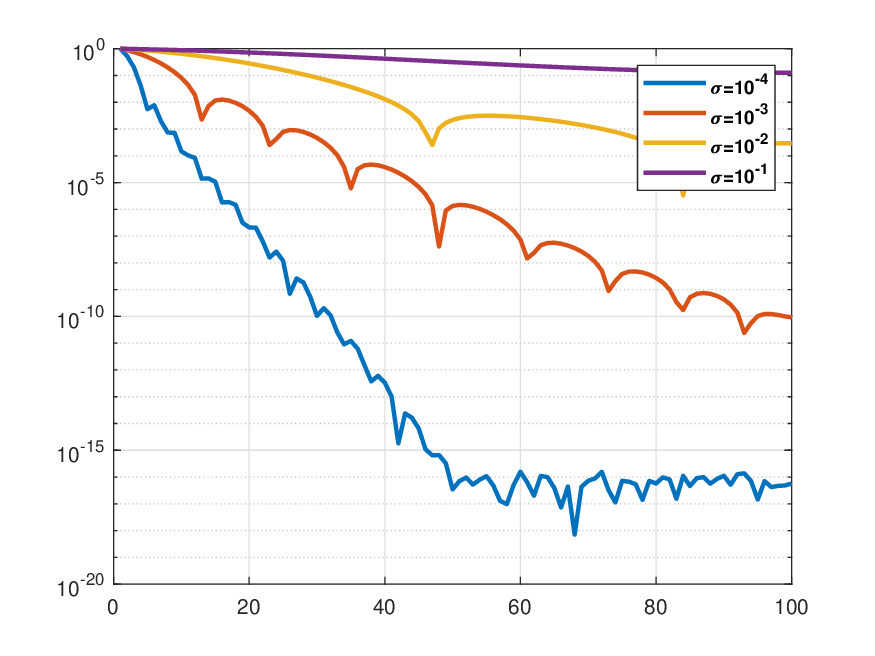}		
 \vspace{-0.54cm}
\caption{Decay  in logarithmic scale of the first column of $|P(y_0)|$ varying the parameter $r$.}
\label{Fig_sdre}
\end{figure}
In order to investigate numerically the error of a separable approximation, we consider a family of $r$-banded matrices $\{P_r(y)\}_r$ for every $y \in \mathbb{R}^n$ as considered in \eqref{banded} for the matrix $A$. In Table \ref{cost_sdre} we consider the error in the computation of the total cost obtained using the banded matrix $\{P_r(y)\}_r$ instead of $P(y)$ for the computation of the feedback map \eqref{feedback_sdre} varying the bandwidth $r$. Fixing the viscosity $\sigma=10^{-4}$, the fast decay displayed in Fig. \ref{Fig_sdre} is reflected in the fast convergence for the total cost. Indeed, it is enough to consider $r=n/10$ to obtain an error of order $10^{-4}$. On the other hand, for the higher viscosity $\sigma=10^{-3}$ the initial error and the decay are worse, leading to the necessity to consider larger bandwidths, i.e., to increase $r$ and consider larger neighborhoods $\mathcal{B}_r(j)$ in the corresponding graph.
\begin{table}[h]
\vspace{0.0cm}
\centering
\begin{tabular}{c|cc}     
$r$  & $\sigma = 10^{-4}$ & $\sigma = 10^{-3}$  \\ \hline
2  &  1.05e-1 &  99.86\\
5 & 1.20e-3 &  3.16 \\
10 & 4.72e-4 &  9.27e-2 \\
20 &  2.50e-8 & 5.60e-3 \\
 \end{tabular}
 \caption{Error in the computation of the total cost.}
 \vspace{-0.75cm}
 \label{cost_sdre}
\end{table}

\section{Conclusion}
In this paper we presented a separable approximation of optimal value functions under a decaying sensitivity assumption. Beyond the computation of curse-of-dimensionality-free approximations with neural networks, as future research we aim at investigating conditions for a decaying sensitivity in a continuous-time setting as well as the influence of the graph-structure on the sensitivity decay.
{\bibliographystyle{abbrv}
\bibliography{main} 
}

\end{document}